\documentclass[12pt]{article}

\usepackage{a4wide}
\usepackage{amsmath}
\usepackage{amsthm}
\usepackage{amssymb}
\usepackage[english]{babel}
\usepackage{graphicx}

\newcommand{\R}{\mathbb{R}}

\newcommand{\sd}{\textup{sd}} 
\newcommand{\st}{\textup{st}} 
\newcommand{\ord}{\textup{ord}} 
\newcommand{\ind}{\textup{ind}} 
\newcommand{\osigma}{\stackrel{\circ}{\sigma}} 
\newcommand{\otheta}{\stackrel{\circ}{\theta}} 
\newcommand{\dtau}{\dot{\tau}} 
 
\newcommand{\lk}{\textup{lk}}

\DeclareMathAlphabet{\mathcalligra}{T1}{calligra}{m}{n}

\DeclareFontFamily{U}{mathx}{\hyphenchar\font45}
\DeclareFontShape{U}{mathx}{m}{n}{
      <5> <6> <7> <8> <9> <10>
      <10.95> <12> <14.4> <17.28> <20.74> <24.88>
      mathx10
      }{}
\DeclareSymbolFont{mathx}{U}{mathx}{m}{n}
\DeclareFontSubstitution{U}{mathx}{m}{n}
\DeclareMathAccent{\widecheck}{0}{mathx}{"71}

\newtheorem{theo}{Theorem}[section]
\newtheorem{lemma}[theo]{Lemma}
\newtheorem{prop}[theo]{Proposition}
\newtheorem{cor}[theo]{Corollary}
\newtheorem{rem}[theo]{Remark}
\newtheorem{ex}[theo]{Example}

\newtheorem{defi}[theo]{Definition}

\begin{document}
\title{Morse shellings out of discrete Morse functions}
\author{Jean-Yves Welschinger}
\maketitle

\begin{abstract} 
\vspace{0.5cm}

From the topological viewpoint, Morse shellings of finite simplicial complexes are {\it pinched} handle decompositions and extend the classical shellings. We prove that every discrete Morse function on a finite simplicial complex induces Morse shellings on its second barycentric subdivision whose critical tiles  -or pinched handles- are in one-to-one correspondence with the critical faces of the function, preserving the index. The same holds true, given any smooth Morse function on a closed manifold, for any piecewise-linear triangulation on it after sufficiently many barycentric subdivisions. \\


\end{abstract}

\section{Introduction}

Every compact piecewise linear manifold $M$ carries a piecewise linear handle decomposition, that is a filtration $\emptyset \subset M_0 \subset \dots \subset M_N = M$ of compact $PL$-manifolds such that each level $M_p$, $p \in \{1 , \dots , N \}$, is obtained from $M_{p-1}$ by attaching a $PL$-handle on its boundary using some piecewise linear homeomorphism, see \cite{RS} and \S \ref{subsechandledecomp}. When the manifold is smooth, such a decomposition can be deduced from the sublevel sets of any smooth Morse function and defined in the smooth category, see \cite{Morse, Smale, Milnor}. When the manifold is combinatorial, that is equipped with a piecewise linear triangulation, see \S \ref{subsecstar}, such a $PL$-handle decomposition can be obtained out of the second barycentric subdivision of the latter, where the handles are derived neighborhoods of the barycenters of the simplices, see Definition \ref{defderived}, Proposition $6.9$ of \cite{RS} and \S \ref{subsechandledecomp}. However, a $PL$-handle decomposition does not carry any canonical triangulation and in fact, triangulated manifolds which are not combinatorial do not carry any such decomposition in the $PL$-category, for the underlying piecewise linear structure would then be that of a $PL$-manifold, see \cite{RS,Qui}. Closed topological manifolds do carry some handle decompositions in the topological category though, except when four-dimensional and non smooth, see \cite{KS,FQ}.
We recently introduced some counterparts to handle decompositions in the simplicial category, which we called shelled $h$-tilings and Morse tilings due to their relations with the classical combinatorial notions of shellings \cite{BruMan,HachZ,Stanbook,Z,RS}, $h$-vectors \cite{Stan,BilLee,Fult,Stan2,Stan3} and discrete Morse theory \cite{For1,For2,Ben}, see Definition \ref{deftilings} and \cite{WelAdv,SaWel2}. Such shelled tilings exist on all finite simplicial complex after finitely many stellar subdivisions at facets \cite{WelAdv}, they encode a class of compatible discrete Morse functions whose critical faces are in one-to-one correspondence with the critical tiles -or pinched handles- of the tiling, preserving the index, see \S \ref{subsectiles} as well as \cite{SaWel2,WelHHA}  and they compute the (co)homology of the complex via two spectral sequences whose first pages are free graded modules over the critical tiles \cite{WelHHA}.

Our main result now goes in the opposite direction and provides Morse shellings out of any discrete Morse function after two barycentric subdivisions. 

\begin{theo}
\label{maintheo}
Let $f$ be a discrete Morse function on a finite simplicial complex $K$.  Then, the second barycentric subdivision of $K$ carries Morse shellings whose critical tiles are in one-to-one correspondence with the critical faces of $f$, preserving the index.
\end{theo}

We recall the definitions of simplicial complexes and discrete Morse functions in \S \ref{subsecprelim} and that of Morse shellings in \S \ref{subsechandledecomp}. In fact, the proof of Theorem \ref{maintheo} provides a relative version of this theorem as well, see Remark \ref{finalremark}. In the case of the trivial discrete Morse function, for which all simplices are critical, Theorem \ref{maintheo} provides a Morse shelling on the second barycentric subdivision whose critical tiles are in one-to-one correspondence with the faces of the complex. It thus gives some simplicial counterpart to the $PL$-handle decomposition of \cite{RS} in the case of combinatorial manifolds. The first barycentric subdivision already inherits Morse shellings by \cite{WelAdv}, but without control on their critical tiles. This however plays a crucial role in the proof of Theorem \ref{maintheo}, in a slightly refined form which we establish in Theorem \ref{theofirst}. 

In the case of smooth Morse functions on smooth manifolds, we deduce the following. 

\begin{cor}
\label{corsmooth}
Let $f$ be a smooth Morse function on a smooth closed manifold $M$ and let $h : K \to M$ be any $PL$-triangulation on $M$. Then, as soon as $d$ is large enough, the $d$-th barycentric subdivision of the simplicial complex $K$ carries Morse shellings whose critical tiles are in one-to-one correspondence with the critical points of $f$, preserving the index.
\end{cor}

Recall that every smooth manifold carries a canonical piecewise linear structure from Whitehead's theorem, see \cite{Whi1} and \S \ref{subsechandledecomp}. We also deduce.

\begin{cor}
\label{corcollapse}
The second barycentric subdivision of any collapsible finite simplicial complex carries Morse shellings using one closed simplex as unique critical tile.
\end{cor}
This result  applies in particular to regular neighborhoods of collapsible subcomplexes in closed triangulated manifolds and provides thus some counterpart to Whitehead's regular neighborhood theorem \cite{Whi,RS} up to which these are $PL$-balls in the case of combinatorial manifolds.

Theorem \ref{maintheo} combined with \cite{SaWel2}, see also Theorem $5.2$ of \cite{WelHHA}, closely relates Morse shellings with discrete Morse functions. In particular, the minimal number of critical points of discrete Morse functions on all barycentric subdivisions of a finite simplicial complex coincides with the minimal number of critical tiles of their Morse shellings. How much of the topology and combinatorics of a simplicial complex does a shelled $h$-tiling encodes remains puzzling, see \S \ref{subsechandledecomp}. Every Morse shelling can be turned into some shelled $h$-tiling after finitely many stellar subdivisions at facets by \cite{WelHHA}, but adding critical tiles in the process. 

We recall in section \ref{sectiles} what we need from the theories of simplicial complexes and discrete Morse functions. We also introduce Morse tiles, describe them as pinched handles and study their structure and first barycentric subdivision. Section \ref{secshellings} is mostly devoted to the proof of Theorem \ref{maintheo} and its corollaries. We first introduce Morse shellings and prove Corollaries \ref{corsmooth}, \ref{corcollapse}. We then prove the existence of specific Morse shellings on first barycentric subdivisions of Morse tiles, see Theorems \ref{theotile1} and \ref{theotile2},  and of simplicial complexes, see Theorem \ref{theofirst}. The latter slightly refines earlier results of \cite{SaWel2, WelAdv} for the need of Theorem \ref{maintheo}.

\section{Morse tiles}
\label{sectiles}

\subsection{Relative simplicial complexes and discrete Morse functions}
\label{subsecprelim}

As in \cite{WelAdv}, let us first recall what we need from the theory of simplicial complexes, see \cite{Koz,Munk}. From the combinatorial point of view, an {\it $n$-simplex} $\sigma$ is a set of cardinality $n+1$ whose elements are called {\it vertices}. Any subset of this finite set, including the empty set, is called a {\it face}. Its {\it geometric realization} is the convex set
$\vert \sigma \vert = \{ \lambda : \sigma \to \R^+ \, \vert \, \sum_{v \in \sigma} \lambda (v) = 1 \}$, it spans the $n$-dimensional real affine space 
$A_{\sigma} = \{ \lambda : \sigma \to \R \, \vert \, \sum_{v \in \sigma} \lambda (v) = 1 \}$. Likewise, from the combinatorial point of view, a {\it finite simplicial complex} $K$ is a collection of subsets of a finite set $V_K$ which contains all singletons and all subsets of its elements. The elements of $K$ are simplices and any simplex defines itself a finite simplicial complex. The  {\it geometric realization} of a finite simplicial complex $K$ is the subset $\vert K \vert = \{ \lambda : V_K \to \R^+ \, \vert \, \sum_{v \in V_K} \lambda (v) = 1 \text{ and } \text{supp} (\lambda) \in K \}$ of $A_{V_K}$, where $\text{supp} (\lambda) = \{ v \in V_K \, \vert \, \lambda (v) \neq 0 \}$. This topological space is then covered by the geometric realizations of all the simplices of $K$ which are maximal with respect to the inclusion, called {\it facets}, and moreover, any two simplices intersect along a unique common face, possibly empty. A face which has codimension one in any facet containing it is called a  {\it ridge}. The {\it dimension} of a simplicial complex $K$  is the maximal dimension of its facets and when they all have same dimension, $K$ is said to be {\it pure dimensional}. 
We denote by $K^{[p]}$ the subset of $p$-simplices of $K$, $p \geq 0$.
When $\vert K \vert $ turns out to be homeomorphic to a closed topological manifold, it is said to be a closed {\it triangulated manifold}.
This condition is less restrictive than that of combinatorial or piecewise-linear manifold, see \S \ref{subsecstar}. Note that any function from $V_K$ to some real affine space $E$ extends to an affine map $A_{V_K} \to E$ which restricts to $\vert K \vert$ and when this restriction is injective, it embeds $\vert K \vert$ into $E$. For example, the boundary of any convex simplicial polytope of $\R^n$ is the geometric realization of a triangulated sphere, embedded into $\R^n$.
This notion of simplicial complex has some relative counterpart introduced by R. Stanley in \cite{Stan3}, see also \cite{Stanbook,CKS}.
A {\it relative simplex} $P$ is a simplex $\overline{P}$ deprived of several of its proper faces $\tau_0, \dots , \tau_k$. A face of $P$ is a relative simplex $\tau \setminus (\tau_0 \cup \dots \cup \tau_k)$, where $\tau$ is a face of its underlying simplex $\overline{P}$ not contained in $\tau_0 \cup \dots \cup \tau_k$, and its dimension is the dimension of $\tau$. 
The geometric realization of $P$ is the complement $\vert P \vert = \vert \overline{P} \vert \setminus (\vert  \tau_0 \vert  \cup \dots \cup \vert  \tau_k \vert )$, while from the combinatorial point of view, $\tau_0, \dots , \tau_k$ and their faces are no more faces of $P$. We will call them the {\it missing faces} of $P$. Some relative simplices are of special interest for us and we devote the remaining part of section \ref{sectiles} to their study, see Definition \ref{deftiles}. 
A {\it relative finite simplicial complex} $S$ is a collection of relative simplices $\{ \sigma \setminus (\sigma \cap L) \, \vert \, \sigma \in K \}$, where $L$ is a subcomplex of a finite simplicial complex $K$. We may assume the subcomplex $L$ of $K$ not to contain any facet of $K$, deleting them from $K$ and $L$ otherwise.
We set $\overline{S}=K$ and call with some abuse this complex the {\it underlying simplicial complex} of $S$, while the faces of $L \cap \overline{S}= \overline{S} \setminus S$ are called the {\it missing faces} of $S$. Contrary to the case of relative simplices, the pair $(K,L)$ such that $S = K \setminus L$ is not unique. The geometric realization of $S$ is the complement $\vert S \vert = \vert \overline{S} \vert \setminus \vert L \vert$. A {\it relative subcomplex} $S'$ of $S$ is a relative complex $K' \setminus L$, where $K'$ is a subcomplex of $K$.

The {\it first barycentric} subdivision $\sd (K)$ of a finite simplicial complex $K$ is a collection of sets $\{ \sigma_0 , \dots , \sigma_q \}$ of elements of $K$ such that $\emptyset \neq \sigma_0 \subsetneq \sigma_1 \subsetneq \dots \subsetneq \sigma_q$, so that $V_{\sd (K)} = K \setminus \{ \emptyset \}$. The map $\sigma \in K \setminus \{ \emptyset \} \mapsto \hat{\sigma} \in \vert K \vert \subset A_{\sigma}$, where $\hat{\sigma}$ denotes the barycenter of $\vert \sigma \vert$, defines by extension an homeomorphic embedding $\vert \sd (K) \vert \to \vert K \vert$, see Proposition $2.33$ of \cite{Koz}. To avoid any confusion, we will denote by $\hat{\sigma}$ the vertex of $\sd (K)$  associated to the simplex $\sigma$ of $K$.
The first barycentric subdivision of a finite relative simplicial complex $S=K \setminus L$ is the relative simplicial complex $\sd (S) = \sd (K) \setminus \sd (L)$. The {\it join} $\sigma_1 * \sigma_2$
of two simplices $\sigma_1 , \sigma_2$  is the simplex $\sigma_1 \cup \sigma_2$. Its geometric realization gets isomorphic to the set of convex combinations $\{ tx_1 + (1-t)x_2 \; \vert \; x_i \in \vert \sigma_i \vert \text{ and } t \in [0,1] \}$, see \cite{RS}. Likewise, the join $K_1 * K_2$
of two finite simplicial complexes is the collection  of joins $\{ \sigma_1 * \sigma_2 \; \vert \; \sigma_i \in K_i \}$, where by convention $\emptyset * \sigma = \sigma * \emptyset = \sigma$. The join $S_1 * S_2$ of two finite relative simplicial complexes $S_1 = K_1 \setminus L_1$ and $S_2 = K_2 \setminus L_2$ is the finite relative simplicial complex $(K_1 * K_2) \setminus \big( (L_1 * K_2) \cup (K_1 * L_2) \big)$. Finally, a subcomplex $S'$ of a finite relative simplicial complex $S$ is an {\it elementary collapse} of $S$ iff $S \setminus S'$ consists of a facet of $S$ together with one of its ridges not contained in any other facet of $S$. It is a {\it collapse} of $S$ if it can be obtained from $S$ after a finite sequence of elementary collapses, see for instance \cite{AdipBen,For1}. This notion is closely related to the discrete Morse theory of R. Forman \cite{For1}. A {\it discrete Morse function} $f$ on a finite simplicial complex $K$ is a function $f : K \setminus \{ \emptyset \} \to \R$ such that for every simplex $\sigma \in K$, the sets $\{ \tau \in K \; \vert \; \sigma \subsetneq \tau \text{ and } f(\sigma) \geq f(\tau) \}$ and $\{ \theta \in K \; \vert \; \theta \subsetneq \sigma \text{ and } f(\sigma) \leq f(\theta) \}$ have cardinalities at most one. Of special interest are the Morse functions which are {\it monotone} in the sense that $f(\sigma) \leq f(\tau) $ whenever $\sigma \subset \tau $, {\it semi-injective} in the sense that all of their preimages have cardinalities at most two and {\it generic} in the sense that an equality $f(\sigma) = f(\tau) $ implies that $\sigma$ is a face of $\tau$ or vice-versa. A simplex $\sigma \in K$ at which such a discrete Morse function $f$ is injective is called a {\it critical face} of $f$ of index $\dim (\sigma)$. The sublevel sets $\{ \sigma \in K \; \vert \; f(\sigma) \leq m \}$, $m \in \R$, are subcomplexes of $K$ and one deduces a sublevel set from the next one either by deleting a critical face, that is removing it from the complex, or by an elementary collapse. It follows from Theorems $3.3$ and $3.4$ of \cite{For1} that every discrete Morse function can be perturbed into such a monotone, semi-injective and generic one, so that one can restricts oneself to the the latter without much loss of generality, as is done in \cite{Ben} for instance. 

\subsection{Morse tiles versus pinched handles}
\label{subsectiles}

The following two families of relative simplices are in the core of our study, they were introduced in \cite{SaWel1, SaWel2}.

\begin{defi}
\label{deftiles}
A basic (resp. Morse) tile of dimension $n$ and order $k \in \{0 , \dots , n+1 \}$ is an $n$-simplex deprived of $k$ ridges (resp. together with possibly a unique face of higher codimension, called its Morse face). 
\end{defi}

If $T$ is a basic tile of order $k = \ord (T)$, then every non-empty face $\mu$ of $T$ has to contain the $(k-1)$-dimensional face $r(T)$, called its {\it restriction set}, whose missing vertices are opposite to the missing ridges of $T$, compare \cite{Stanbook}.  The tile $T \setminus \mu$  is said to be {\it critical}  of index $k = \ind (T \setminus \mu)$ when $\mu = r(T)$ and it is said to be {\it regular} otherwise. The closed simplex deprived of its empty face, denoted by $\dot{\sigma} = \sigma \setminus \{ \emptyset \}$, is thus critical of vanishing index, but we also consider the closed simplex itself as being critical of vanishing index, though it differs from $\dot{\sigma}$ as a relative simplex. As for the open simplex, denoted by $\osigma$, it is critical of index $\dim (\sigma)$ and when the latter vanishes, $\osigma$ and $\dot{\sigma}$ coincide.

 \begin{figure}[h]
   \begin{center}
    \includegraphics[scale=1]{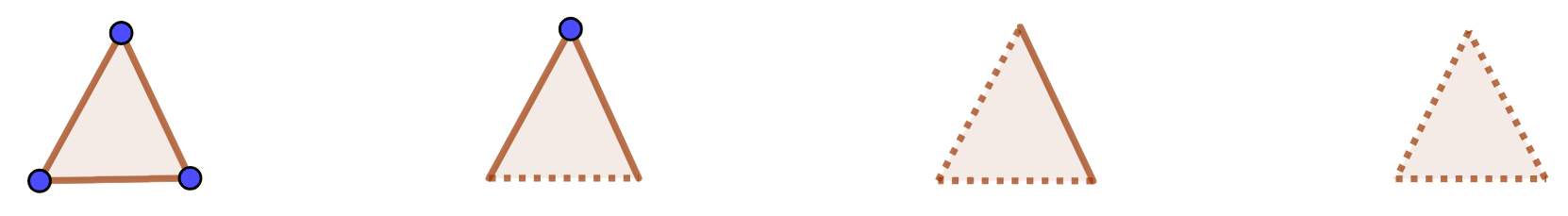}
    \caption{The basic tiles in dimension two.}
    \label{figbasictiles}
      \end{center}
 \end{figure}

 \begin{figure}[h]
   \begin{center}
    \includegraphics[scale=1]{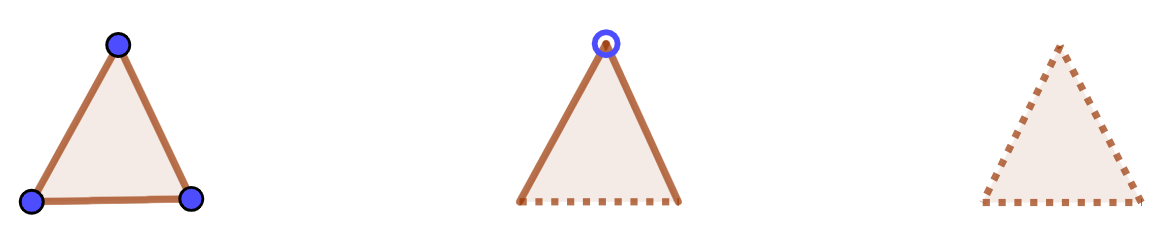}
    \caption{The critical tiles in dimension two.}
    \label{figbasictiles}
      \end{center}
 \end{figure}

From the topological viewpoint, an $n$-dimensional critical tile of index $k \in \{ 0 , \dots  , n\}$ is a simplicial {\it pinched handle} of dimension $n$ and index $k$. It is obtained by pinching onto $\mu$ the missing face $\mu \times \theta$ of a piecewise linear handle $\osigma \times \theta$, where $\sigma$ is a $k$-simplex, $\theta$ an $(n-k)$-simplex and $\mu$ a ridge of $\sigma$. The projection onto the first factor $\sigma \times \theta \to \sigma = \mu * v$, where $v$ denotes the vertex opposite to $\mu$ in $\sigma$, indeed factors through the join $\mu * \theta$, the first arrow of the diagram $\sigma \times \theta \to \mu * \theta \stackrel{p}{\to} \sigma$ pinching $\mu \times \theta$ onto $\mu$. By removing the preimage of the boundary of $\sigma$, we deduce the diagram
$\osigma \times \theta \to C^n_k {\to} \osigma$, where $C^n_k = ( \mu * \theta) \setminus p^{-1} (\partial \sigma)$ is a critical tile of index $k$.

Likewise, an $n$-dimensional regular basic tile is an $n$-simplex deprived of a union of ridges defining an $(n-1)$-dimensional piecewise linear ball embedded in its boundary, whereas for a non-basic regular Morse tile, the latter gets pinched as well. Let indeed $T$ be a non basic regular Morse tile of dimension $n$ and order $k$, whose Morse face $\mu$ has dimension $l \geq k$. Let $\sigma = \mu * v$ be an $(l+1)$-simplex and $\tau$ a union of $k+1$ ridges of $\sigma$ containing $\mu$, so that $\sigma \setminus \tau$ is a regular basic tile. Let finally $\theta$ be an $(n-l-1)$-simplex.
By depriving the preimage of $\tau$ of the diagram $\sigma \times \theta \to \mu * \theta {\to} \sigma$, we deduce the map $(\sigma \setminus \tau) \times \theta \to T$ which pinches the missing face 
$\mu \times \theta$ onto $\mu$, whereas $\tau \times \theta$ is  an $(n-1)$-dimensional $PL$-ball embedded in the boundary of $\sigma \times \theta$.

\begin{lemma}
\label{lemma1}
Let $T$ be a basic tile of order $k$ and $T'$ be a basic (resp. Morse) tile of order $k'$. Then, $T * T'$ is a basic (resp. Morse) tile of order $k+k'$.
\end{lemma}

\begin{proof}
Set $T = \sigma \setminus (\sigma_0 \cup \dots \cup \sigma_{k-1})$ and $T' = \sigma' \setminus (\sigma'_0 \cup \dots \cup \sigma'_{k'-1} \cup \mu)$, where $\sigma_0 , \dots , \sigma_{k-1}$ (resp. 
$\sigma'_0 , \dots , \sigma'_{k'-1}$) are ridges of the simplex $\sigma$ (resp. $\sigma'$) and $\mu$ is a higher codimensional face of $\sigma'$, possibly empty. Then, 
$T * T' = (\sigma * \sigma') \setminus \big( \cup_{i=0}^{k-1} (\sigma _i* \sigma') \cup  \cup_{j=0}^{k'-1}  (\sigma * \sigma'_j)  \cup (\sigma * \mu) \big)$ is a Morse tile of order $k+k'$ which is basic if and only if $\mu$ is empty, that is iff $T'$ is basic.
\end{proof}
For instance, the cone $v * T$ with apex $v$ over a basic tile $T$ of order $k$ is itself a basic tile of order $k$, whereas the cone deprived of its basis $\dot{v} * T$  is a basic tile of order $k+1$.

\begin{prop}
\label{prop1}
Every Morse tile splits uniquely as a join $\sigma * \otheta * \dtau$, where $\sigma$ is a closed simplex, $\otheta$ an open one and $\dtau$ a closed simplex of non-vanishing dimension deprived of its empty face. Conversely, every such join is a Morse tile, which is critical different from a closed simplex (resp. basic) if and only if $\sigma$ (resp. $\tau$) is empty. 
\end{prop}
Each factor of the decomposition $\sigma * \otheta * \dtau$ of a Morse tile may be empty, that is missing. If the tile is a closed simplex, an open one or a dotted one for instance, only one factor enters the decomposition. The condition on the dimension of $\tau$ is required to get the uniqueness of the decomposition. Observe that the index of $\otheta$ as a critical tile is $\dim (\theta)$, whereas the one of $\otheta * \dtau$ is $\dim (\theta) + 1$ as soon as $\tau$ is not empty. These two formulas coincide though when $\tau$ is empty if one writes $\otheta = \stackrel{\circ}{\theta'} * \dtau'$, with $\dim (\tau')=0$.

\begin{proof}
Let $T$ be a relative simplex which splits as a join $\sigma * \otheta * \dtau$ and let  $\theta_0 , \dots , \theta_{k-1}$  be the ridges of $\theta$, so that $k = \dim (\theta) + 1$. Then, if $\tau$ is not empty, $T = (\sigma * \theta * \tau) \setminus \big( \cup_{i=0}^{k-1} (\sigma * \theta_i * \tau) \cup (\sigma * \theta) \big)$, where $\sigma * \theta = \{ \emptyset \}$ if $\sigma$ and $\theta$ are empty, so that it is a non-basic Morse tile of order $k$ with Morse face $\sigma * \otheta$. It is thus critical in this case if and only if $\dim (\sigma * \otheta) = k-1$, that is iff $\sigma$ is empty. If $\tau$ is empty, $T = (\sigma * \theta) \setminus \big( \cup_{i=0}^{k-1} (\sigma * \theta_i) \big)$ is a basic tile of order $k$ which is thus critical iff either $\sigma$ or $\theta$ is empty.

Conversely, let $T$ be an $n$-dimensional Morse tile of order $k$, so that $T = \lambda \setminus (\lambda_0 \cup \dots \cup \lambda_{k-1} \cup \mu)$, where $\lambda_0 , \dots , \lambda_{k-1}$ are ridges of the $n$-simplex $\lambda$ and $\mu$ some higher codimension face of $\lambda$, possibly empty. Let $\otheta = r(T)$ be the restriction set of $T$, whose vertices are the vertices opposite to $\lambda_0 , \dots , \lambda_{k-1}$. The face $\mu$, provided it is not empty, contains $r(T)$ and can thus be written $\sigma * \otheta$ for some face $\sigma$ of $\lambda_0 \cap \dots \cap \lambda_{k-1}$, possibly empty. Denoting by $\tau$ the link of $\sigma * \theta$ in $\lambda$, that is the face whose vertices are the ones not contained in $\sigma * \theta$, we deduce that
$T = \sigma * \otheta * \dtau$. This decomposition has to be unique, since $\otheta = r(T)$ and $\mu = \sigma * \otheta$.
\end{proof}

\begin{rem}
\label{remrt}
We saw in the proof of Proposition \ref{prop1} that the uniqueness of the decomposition $T = \sigma * \otheta * \dtau$ follows from the fact that $\otheta$ has to be the restriction set of $T$ and $\sigma * \otheta$ has to be its Morse face as soon as $\tau$ is not empty, that is as soon as $T$ is not basic.
\end{rem}

\begin{cor}
\label{cor1}
Let $T$ be a Morse tile and $v$ a vertex. Then, the cone $v*T$ is a closed simplex if and only if $T$ is a closed simplex and a regular Morse tile otherwise. Likewise, the cone deprived of its basis $\dot{v} * T$ is a Morse tile which is critical iff $T$ is critical different from a closed simplex and its index equals then $\ind (T) + 1$.
\end{cor}

\begin{proof}
By Proposition \ref{prop1}, $T$ splits as a join $\sigma * \otheta * \dtau$, so that $v*T = (v * \sigma) * \otheta * \dtau$ is a closed simplex if and only if $\theta$ and $\tau$ are empty and a regular Morse tile otherwise, since $v * \sigma \neq \emptyset$. Likewise, $\dot{v} * T = \sigma * (\dot{v} *   \otheta) * \dtau$ is critical if and only if $\sigma = \emptyset$ and its index equals then
$\dim (\dot{v} *   \otheta)$ if $\tau$ is empty and $\dim (\dot{v} *   \otheta) + 1$ otherwise, that is $\ind (T) + 1$ in any case.
\end{proof}

\subsection{Star neighborhoods}
\label{subsecstar}

Recall that the {\it star neighborhood} of every face $\sigma$ of a simplicial complex $K$ is the subcomplex $\st_K (\sigma) = \{ \tau \in K \; \vert \; \sigma * \tau \in K \}$, that is the smallest subcomplex of $K$ which is a neighborhood of $\sigma$. The {\it link} of $\sigma$ in $K$ equals $\lk_K (\sigma) = \{ \tau \in \st_K (\sigma) \; \vert \; \sigma \cap \tau = \emptyset \}$, so that
$\lk_K (\sigma) = \{ \emptyset \}$ when $\st_K (\sigma) = \sigma$ and in general, $\st_K (\sigma) = \sigma * \lk_K (\sigma) $, see \cite{RS}. 

We then recall that a closed {\it combinatorial manifold} is a finite simplicial complex whose links share common subdivisions with boundaries of simplices, see \cite{Lick2}. Every link is thus a combinatorial sphere of lower dimension. The underlying topological space of a combinatorial manifold inherits a piecewise linear structure which is that of a {\it piecewise linear manifold}, that is of a topological manifold whose change of coordinates are piecewise linear homeomorphisms, see \cite{RS,Qui}. Not all triangulated manifolds are combinatorial \cite{Can,Edw}.

Let us finally recall for the reader's convenience the structure of links in the first and second barycentric subdivisions of $K$, see Proposition $6.9$ of \cite{RS} and \cite{AdipBen2}.

\begin{prop}
\label{prop2}
Let $K$ be a finite simplicial complex. Then,

1) For every face $\sigma$ of $K$, the map $\lambda \in  \lk_K (\sigma) \mapsto \sigma * \lambda \in \st_K (\sigma)$ induces the isomorphism of simplicial complexes $\sd (\partial \sigma) * \sd (
\lk_K (\sigma) ) \to \lk_{\sd (K)} (\hat{\sigma}) $, where $\sd (\partial \sigma) = \{ \emptyset \}$ if $\dim (\sigma )=0$.

2) As soon as $K$ is not empty, $\sd (K) = \bigcup_{v \in K^{[0]}} \st_{\sd (K)} (\hat{v})$, where the union is taken over all vertices of $K$.

3) For every vertices $v_1 \neq v_2$ of $K$, $\st_{\sd (K)} (\hat{v}_1) \cap \st_{\sd (K)} (\hat{v}_2) = \lk_{\sd (K)} (\hat{v}_1) \cap \lk_{\sd (K)} (\hat{v}_2) $ is not empty if and only if $v_1$ and $v_2$ share a common edge of $K$ and in this case, this intersection is the image of $\st_{\sd (\lk_K (v_1))} (\hat{v}_2) $ (resp. $\st_{\sd (\lk_K (v_2))} (\hat{v}_1) $ ) in $\lk_{\sd (K)} (\hat{v}_1)$
(resp. $\lk_{\sd (K)} (\hat{v}_2)$) under the isomorphism $1)$.
\end{prop}

\begin{proof}
1) By definition, $\lk_{\sd (K)} (\hat{\sigma}) $ is the collection of simplices $[\hat{\sigma}_0, \dots , \hat{\sigma}_{i-1}, \hat{\sigma}_{i+1}, \dots , \hat{\sigma}_p]$ of $\sd (K)$ such that
$\sigma_0 \subset \dots \subset \sigma_{i-1} \subset \sigma_{i+1} \subset \dots \subset \sigma_p$ is a flag of simplices of $K$ and $ \sigma_{i-1} \subset \sigma \subset \sigma_{i+1} $.
Such a simplex is thus a join $[\hat{\sigma}_0, \dots , \hat{\sigma}_{i-1}] * [ \hat{\sigma}_{i+1}, \dots , \hat{\sigma}_p]$, where $[\hat{\sigma}_0, \dots , \hat{\sigma}_{i-1}]  \in \sd (\partial \sigma)$ and 
$[ \hat{\sigma}_{i+1}, \dots , \hat{\sigma}_p]$ is in bijective correspondence with $[\lk_{{\sigma}_{i+1}} (\sigma) , \dots , \lk_{{\sigma}_{p}} (\sigma) ] \in \sd ( \lk_K (\sigma) ) $ via the map $\lambda \in  \lk_K (\sigma) \mapsto \sigma * \lambda \in \st_K (\sigma)$. The first part of the proposition follows.

2) By definition, every simplex of $\sd (K)$ reads $[\hat{\sigma}_0, \dots , \hat{\sigma}_{p}]$, where $\emptyset \neq \sigma_0 \subset \dots \subset \sigma_p$ is a flag of $K$. Such a simplex thus belongs to 
$\st_{\sd (K)} (\hat{v})$ for every vertex $v$ of ${\sigma}_0$.

3) It follows that if $v_1 \neq v_2$ are vertices of $K$, $\st_{\sd (K)} (\hat{v}_1) \cap \st_{\sd (K)} (\hat{v}_2) $ is the set of simplices $[\hat{\sigma}_0, \dots , \hat{\sigma}_{p}]$ of  $\sd (K)$ such that both $v_1$ and $v_2$ are contained in ${\sigma}_0$, so that this intersection coincides with $\lk_{\sd (K)} (\hat{v}_1) \cap \lk_{\sd (K)} (\hat{v}_2) $  and that its preimage under the isomorphism $1)$ 
equals $\st_{\sd (\lk_K (v_1))} (\hat{v}_2) $ (resp. $\st_{\sd (\lk_K (v_2))} (\hat{v}_1) $) in $ \sd ( \lk_K (v_1) )$ (resp. $ \sd ( \lk_K (v_2) )$).
\end{proof}

Following \cite{AdipBen2, RS}, we adopt the following definition.

\begin{defi}
\label{defderived}
The first derived neighborhood of a subcomplex $L$ in a simplicial complex $K$ is the neighborhood $N(L,K) = \bigcup_{v \in L^{[0]}} \st_{\sd (K)} (\hat{v})$ of $\sd (L)$ in $\sd (K)$.
\end{defi}

\begin{cor}
\label{cor2}
Let $K$ be a finite simplicial complex. Then,

1) For every simplex $\sigma$ of $K$, the map $\lambda \in  \lk_K (\sigma) \mapsto \sigma * \lambda \in \st_K (\sigma)$ induces the isomorphism of simplicial complexes $\sd \big( \sd (\partial \sigma) * \sd ( \lk_K (\sigma) ) \big) \to \lk_{\sd^2 (K)} (\hat{\hat{\sigma}}) $.

2) For every simplices $\sigma \neq \tau$ of $K$, $\st_{\sd^2 (K)} (\hat{\hat{\sigma}}) \cap \st_{\sd^2 (K)} (\hat{\hat{\tau}}) = \lk_{\sd^2 (K)} (\hat{\hat{\sigma}}) \cap \lk_{\sd^2 (K)} (\hat{\hat{\tau}}) $ is not empty if and only if $\sigma$ is a face of  $\tau$ or vice-versa. Moreover, this intersection is the image of $N \big(\hat{\tau} , \sd (\partial \sigma) * \sd (\lk_K (\sigma)) \big)$ (resp.  $N \big( \hat{\sigma} , \sd (\partial \tau) * \sd (\lk_K (\tau))\big)$) in $\lk_{\sd^2 (K)} (\hat{\hat{\sigma}})$ (resp. $\lk_{\sd^2 (K)} (\hat{\hat{\tau}})$) under the isomorphism $1)$.
\end{cor}

In the second part of Corollary \ref{cor2}, if $\sigma$ is a face of $\tau$ for instance, $\hat{\sigma}$ is a vertex of $ \sd (\partial \tau) $, subcomplex of $ \sd (\partial \tau) * \sd (\lk_K (\tau))$, whereas 
$\hat{\tau}$ denotes the vertex $\widehat{\lk_\tau (\sigma)}$ of $\sd (\lk_K (\sigma))$ to which it gets identified by the first part of Proposition \ref{prop2}. 

\begin{proof}
1) The first part of Proposition \ref{prop2} applied to the vertex $\hat{\sigma}$ of $\sd (K)$ provides the isomorphism of simplicial complexes $\sd (\lk_{\sd (K)} (\hat{\sigma})) \to \lk_{\sd^2 (K)} (\hat{\hat{\sigma}}) $. A second application of Proposition \ref{prop2} to the simplex $\sigma$ of $K$  provides the isomorphism $\sd (\partial \sigma) * \sd (
\lk_K (\sigma) ) \to \lk_{\sd (K)} (\hat{\sigma}) $. The result follows by composition of these isomorphisms.

2) Likewise, by the third  part of Proposition \ref{prop2}, $\st_{\sd^2 (K)} (\hat{\hat{\sigma}}) \cap \st_{\sd^2 (K)} (\hat{\hat{\tau}}) = \lk_{\sd^2 (K)} (\hat{\hat{\sigma}}) \cap \lk_{\sd^2 (K)} (\hat{\hat{\tau}}) $ is non empty if and only if $\hat{\sigma}$ and $\hat{\tau}$  share a common edge of $\sd (K)$, that is iff $\sigma$ is a face of  $\tau$ or vice-versa. Moreover, in this case,
this intersection gets identified with $\st_{\sd (\lk_{\sd (K)} (\hat{\sigma}))} (\hat{\hat{\tau}})$ (resp. $\st_{\sd (\lk_{\sd (K)} (\hat{\tau}))} (\hat{\hat{\sigma}})$) in $ \lk_{\sd^2 (K)} (\hat{\hat{\sigma}})$ (resp. $ \lk_{\sd^2 (K)} (\hat{\hat{\tau}})$), that is, by the first part of Proposition \ref{prop2} and Definition \ref{defderived}, with $N \big(\hat{\tau} , \sd (\partial \sigma) * \sd (\lk_K (\sigma)) \big)$ (resp.  $N \big( \hat{\sigma} , \sd (\partial \tau) * \sd (\lk_K (\tau))\big)$).
\end{proof}

Let now $S = K \setminus L$ be a relative simplicial complex. We set, for every simplex $\sigma$ of $K$, $\st_S (\sigma) = \st_K (\sigma) \setminus \st_L (\sigma) $, where $\st_L (\sigma) = \emptyset$ if $\sigma \notin L$. We then set $\lk_S (\sigma) = \lk_K (\sigma) \setminus \lk_L (\sigma) $, so that $\st_S  (\sigma) = \sigma * \lk_S (\sigma)$. These links and star neighborhoods are thus defined even for the missing faces of $S$, but they depend on the decomposition $S = K \setminus L$, so that they are rather defined for pairs $(K,L)$ than for complements $K \setminus L$.

\begin{ex}
\label{ex1}
If $\otheta$ is an open simplex of positive dimension and $v$ a vertex of $\theta$, then $\theta = v * \theta'$, where $\theta' = \lk_\theta (v)$, but $\st_{\otheta} (v) = v * \stackrel{\circ}{\theta'}$ differs from $\otheta$ and $\lk_{\otheta} (v) = \stackrel{\circ}{\theta'}$.
\end{ex}

In Example \ref{ex1}, neither the link $\lk_{\otheta} (v) $ nor the star neighborhood $\st_{\otheta} (v)$ detects that the face $\theta'$ opposite to $v$ in $\theta$ is missing in $\otheta$.

\begin{ex}
\label{ex2}
Let $T = \theta \setminus v$ be a regular Morse tile, where $\theta$ is a closed simplex of positive dimension and $v$ one of its vertex, so that $\theta = v * \theta'$ with $\theta' = \lk_\theta (v)$.
Then $\st_T (v) = T$ and $\lk_T (v) = \dot{\theta}'$, since $\lk_v (v) =\{ \emptyset \}$.
\end{ex}

In Example \ref{ex2} on the contrary, the link $\lk_T (v)$ detects the fact that the opposite vertex $v$ is missing in $T$. In general, the stars and links of vertices in Morse tiles are given by the following lemma, see Proposition \ref{prop1}. 

\begin{lemma}
\label{lemma2}
Let $T= \sigma * \otheta * \dtau$ be a Morse tile. Then,

1) For every vertex $v$ of $\sigma$, $\lk_T (v) = \lk_\sigma (v) * \otheta * \dtau$.

2) For every vertex $v$ of $\theta$, $\lk_T (v) = \sigma * \lk_{\otheta} (v)  * \dtau$, where $\lk_{\otheta} (v) =\{ \emptyset \}$ if $\dim (\theta) = 0$.

3) For every vertex $v$ of $\tau$, $\lk_T (v) = \sigma * \otheta *  \lk_{\tau} (v) $.
\end{lemma}

As in Example \ref{ex1}, the link given in the second (resp. third) part of Lemma \ref{lemma2} does not detect the fact that the face opposite to $v$ (resp. the Morse face) is missing in $T$.
The link $ \lk_{\otheta} (v) $ is given by Example \ref{ex1}.

\begin{proof}
Let us write $T = K \setminus L$, with $K = \sigma * \theta * \tau$ and $L = (\sigma * \partial \theta * \tau) \cup ( \sigma * \theta )$ if $\tau$ is not empty and $T$ not reduced to $\dtau$, whereas
$L = \{ \emptyset \}$ if $T = \dtau$ and  $L = \sigma * \partial \theta$ if $\tau$ is empty. By definition, if $v \in \sigma$ (resp. $v \in \theta$, $v \in \tau$), then $\lk_K (v) =  \lk_\sigma (v) * \theta * \tau$
(resp. $\lk_K (v) =   \sigma *  \lk_\theta (v)  * \tau$, $\lk_K (v) =   \sigma * \theta * \lk_\tau (v) $). Likewise, if $v \in \sigma$ (resp. $v \in \tau$), then $\lk_L (v) = (\lk_\sigma (v) * \partial \theta * \tau) \cup (\lk_\sigma (v) * \theta)$ (resp. $\lk_L (v) = \sigma * \partial \theta * \lk_\tau (v)$), whereas if $v \in \theta$ and if $\theta' = \lk_\theta (v)$, then $\lk_L (v) = (\sigma * \partial \theta' * \tau)
\cup (\sigma * \theta' )$, see Example \ref{ex1}. The result thus follows from the definition $\lk_T (v) = \lk_K (v) \setminus \lk_L (v) $.
\end{proof}

The stars and links of vertices in the first barycentric subdivisions of Morse tiles are then given by the following corollary, see Proposition \ref{prop1}. 

\begin{cor}
\label{cor3}
Let $T= \sigma * \otheta * \dtau$ be a Morse tile. Then,

1) For every vertex $v$ of $\sigma * \theta * \tau$, the map $\lambda \in  \lk_T (v) \mapsto v * \lambda \in \st_T (v)$ induces the isomorphism of relative simplicial complexes $ \sd (\lk_T (v)) \to \lk_{\sd (T)} (\hat{v}) $.

2) If $T$ is not reduced to $\dtau$, $\sd (T) = \bigcup_{v} \st_{\sd (T)} (\hat{v})$, where the union is taken over all vertices of $\sigma * \theta * \tau$ and deprived of the empty face if $T = \dtau$.

3) For every vertices $v_1 \neq v_2$ of $\sigma * \theta * \tau$, $\st_{\sd (T)} (\hat{v}_1) \cap \st_{\sd (T)} (\hat{v}_2) = \lk_{\sd (T)} (\hat{v}_1) \cap \lk_{\sd (T)} (\hat{v}_2) $ and this intersection is the image of $\st_{\sd (\lk_T (v_1))} (\hat{v}_2) $ (resp. $\st_{\sd (\lk_T (v_2))} (\hat{v}_1) $ ) in $\lk_{\sd (T)} (\hat{v}_1)$
(resp. $\lk_{\sd (T)} (\hat{v}_2)$) under the isomorphism $1)$.
\end{cor}

\begin{proof}
Let $K = \sigma * \theta * \tau$ and $L = (\sigma * \partial \theta * \tau) \cup ( \sigma * \theta )$ if $\tau$ is not empty and $T \neq \dtau$, $L = \{ \emptyset \}$ if $T = \dtau$ and  $L = \sigma * \partial \theta$ if $\tau$ is empty, so that $T = K \setminus L$ and $\sd(T) = \sd(K) \setminus \sd(L)$.

1) Let $v$ be a vertex of $K$. By Proposition \ref{prop2}, the map $\lambda \in  (\lk_K (v), \lk_L (v)) \mapsto v * \lambda \in (\st_K (v) , \st_L (v))$ induces an isomorphism of pairs
$ (\sd (\lk_K (v)) , \sd (\lk_L (v)))\to (\lk_{\sd (K)} (\hat{v}) , \lk_{\sd (L)} (\hat{v})) $. Hence the isomorphism $ \sd (\lk_T (v)) = \sd (\lk_K (v)) \setminus \sd (\lk_L (v)) \to \lk_{\sd (T)} (\hat{v}) =  \lk_{\sd (K)} (\hat{v}) \setminus \lk_{\sd (L)} (\hat{v})$.

2) By Proposition \ref{prop2}, $\sd (K) = \bigcup_{v \in K^{[0]}} \st_{\sd (K)} (\hat{v})$ and as soon as $L$ is not reduced to $ \{ \emptyset \}$, $\sd (L) = \bigcup_{v \in L^{[0]}} \st_{\sd (L)} (\hat{v})$.
We deduce that $\sd (T) = \sd (K) \setminus \sd (L) = \bigcup_{v \in K^{[0]}} \st_{\sd (T)} (\hat{v})$ whenever $L$ is not reduced to $ \{ \emptyset \}$.

3) The third part follows from Proposition \ref{prop2} along the same lines.
\end{proof}

\section{Morse shellings of barycentric subdivisions}
\label{secshellings}

\subsection{Morse shellings versus pinched handle decompositions}
\label{subsechandledecomp}

Let us now define the key notion of the paper, which we gradually introduced in  \cite{SaWel1,SaWel2,WelAdv} under more restrictive forms.

\begin{defi}
\label{deftilings}
A tiling of a relative simplicial complex $S$ is a partition of its geometric realization by relative facets. It is shellable iff it admits a filtration $\emptyset = S_0 \subset S_1 \subset \dots \subset S_N = S$ by relative subcomplexes, called a shelling, such that for every $p \in \{1 , \dots , N\}$, $S_p \setminus S_{p-1}$ consists of a single relative facet of the tiling. It is said to be an $h$-tiling (resp. a Morse tiling) iff all the relative facets involved are basic or critical tiles (resp. Morse tiles) given by Definition \ref{deftiles}.
\end{defi}
If $S = K \setminus L$, where $L$ is a subcomplex of a finite simplicial complex $K$ containing no facet of $K$, then the relative simplices involved in any tiling of $S$ are of the form $\sigma \setminus \tau$, where $\sigma$ is a facet of $K$ and $\tau$ a subcomplex of $\sigma$ containing $\sigma \cap L$. Classical shellings \cite{BruMan,HachZ,AdipBen2,Z,RS} are shelled h-tilings using only basic tiles. Recall that the boundary of any convex simplicial polytope is shellable in the classical sense \cite{BruMan}, but a compact triangulated manifold, in order to be shellable, has to be piecewise-linearly homeomorphic to a simplex or its boundary, see \cite{Koz,RS}. Moreover, many triangulated three-spheres are not shellable and cannot be as soon as they contain a knotted triangle in their one skeleton by \cite{HachZ}, see also \cite{Lick}. In contrast, we proved the following existence result in \cite{WelAdv} and refer to this paper for the classical definition of stellar subdivision.

\begin{theo}[Theorem $1.3$ of \cite{WelAdv}]
\label{theostellar}
Every finite relative simplicial complex carries a shellable $h$-tiling after finitely many stellar subdivisions at facets.
Moreover, the same holds true using stellar subdivisions at ridges instead, or also using mixed ones.
Finally, in bounded dimension, both the sequence of subdivisions and the shelling are given by some quadratic time  algorithm.
\end{theo}
The definition of tiling used in \cite{WelAdv} to get Theorem \ref{theostellar} was actually even a bit more restrictive than that of \ref{deftilings}, namely it required that for every $d \geq 0$, the union of tiles of dimensions greater than $d$ is a subcomplex of the relative simplicial complex, which forces these tiles to be relative facets. In the case of a shelled $h$-tiling using only basic tiles, this stronger condition gets always satisfied but in general it does not and the Morse tilings given by Theorem \ref{maintheo} may not satisfy this stronger condition as well. 

From the topological viewpoint, the filtration given by any shelled $h$-tiling provides a {\it pinched handle decomposition} of the relative simplicial complex, see \S \ref{subsectiles}. Recall that every closed piecewise-linear manifold $M$ carries some piecewise linear handle decomposition, that is some filtration $\emptyset = M_0 \subset M_1 \subset \dots \subset M_N=M$ by compact $PL$-submanifolds such that for every $p \in \{ 1, \dots , N \}$, the manifold $M_p$ is the union of $M_{p-1}$ with a handle $H_p$, the latter being a product of simplices $\sigma \times \theta$ attached to the boundary of $M_{p-1}$ along $\partial \sigma \times \theta$ using some piecewise linear homeomorphism, see \cite{RS}. If $M$ is equipped with some triangulation compatible with its $PL$-structure, that is with some $PL$-homeomorphism $h : K \to M$ from some finite simplicial complex $K$, then an explicit $PL$-handle decomposition of $M$ is given by the union $\cup_{\sigma \in K} \st_{\sd^2 (K)} (\hat{\hat{\sigma}})$, where the filtration is obtained by ordering the simplices of $K$ in increasing dimensions, see Proposition $6.9$ of \cite{RS}. Conversely, given that every handle is attached using some $PL$-homeomorphism, the manifold obtained after a sequence of handle attachments inherits the structure of a piecewise linear manifold, so that a triangulated manifold which is not  combinatorial, that is whose underlying piecewise linear structure is not that of a piecewise linear manifold, cannot carry such a $PL$-handle decomposition. Moreover, a manifold obtained as a sequence of $PL$-handle attachments does not inherit any specific triangulation from the handle attachments, for the handles are not triangulated and the attaching homeomorphisms are not simplicial with respect to any natural triangulation. Theorem \ref{theostellar} on the contrary provides some pinched handle decomposition for every triangulated manifold, even every finite simplicial complex, after finitely many stellar subdivisions at facets. Moreover, the latter is produced in the category of relative simplicial complexes, so that it carries some triangulation. The price to pay for a triangulated manifold is the need to transit through finite simplicial complexes, say singular triangulated manifolds, along the decomposition. The pinched handle decomposition given by any shelled $h$-tiling indeed provides a filtration $\emptyset = M_0 \subset M_1 \subset \dots \subset M_N=M$  where for every $p \in \{ 1, \dots , N \}$, the simplicial complex $M_p$ is obtained from $M_{p-1}$ by either attaching a pinched handle or a regular basic tile by some simplicial embedding from the missing subcomplex of the tile to $M_{p-1}$. In the case of a regular basic tile $T_p$, such an attachment is called an elementary shelling and does not change the $PL$-homeomorphism type provided that $M_{p-1}$ is locally collarable near the union of ridges where $T_p$ is attached, which consists of a $PL$-ball, see \cite{RS}. In the case of a Morse shelling, the attachment of a regular basic tile is replaced by the attachment of a regular Morse tile, that is of a simplex along a pinched $PL$-ball of its boundary, see \S \ref{subsectiles}.

A finite simplicial complex may carry shellable Morse tilings but no $h$-tiling, as for instance the union of two triangles sharing a common vertex. Nevertheless, every Morse tiling (resp. shellable Morse tiling) can be turned into some $h$-tiling (resp. shellable $h$-tiling) after finitely many stellar subdivisions at facets, see \cite{WelAdv}.

Let us now prove Corollaries \ref{corsmooth} and \ref{corcollapse}, while the remaining part of the paper is devoted to the proof of Theorem \ref{maintheo}. Note that in the case of a smooth closed manifold of dimension at most three, the existence of a triangulation carrying a Morse shelling satisfying the conditions of Corollary \ref{corsmooth} has already been established in \cite{SaWel2} using different methods. However, even in these dimensions, Corollary \ref{corsmooth} is stronger since it holds true for every triangulations after sufficiently many barycentric subdivisions, compare \cite{Ben}.

\begin{proof}[Proof of Corollary \ref{corsmooth}]
We know from Theorem $A$ of \cite{Ben} that provided $d$ is large enough, $\sd^d (K)$ carries a discrete Morse function whose critical faces are in one-to-one correspondence with the critical points of $f$, preserving the index, see also \cite{Gal}. We then deduce from Theorem \ref{maintheo} that $\sd^{d+2 }(K)$ carries the desired Morse shelling. 
\end{proof}

\begin{proof}[Proof of Corollary \ref{corcollapse}]
By Lemma $4.3$ of \cite{For1}, every collapsible finite simplicial complex carries a discrete Morse function having a single critical face, of vanishing index, see also \S \ref{subsectiles}. The result thus follows from Theorem \ref{maintheo}.
\end{proof}

We already observed that Corollary \ref{corcollapse} applies in particular to regular neighborhoods of collapsible subcomplexes in closed triangulated manifolds. The latter has to be piecewise linearly homeomorphic to a simplex in the case of combinatorial manifolds by Whitehead's regular neighborhood theorem \cite{Whi,RS}, but is not $PL$-homeomorphic, not even homeomorphic to a ball in general, see \cite{AdipBen2}. In particular, one cannot replace Morse shellings by shelled $h$-tilings in Corollary \ref{corcollapse}, since a pure-dimensional simplicial complex which carries a shelled $h$-tiling using a single closed simplex as critical tile is shellable in the classical sense and thus $PL$-homeomorphic to a simplex, see \cite{Koz,RS}. Shellable $h$-tilings thus encode in a more puzzling way the topology and combinatorics of triangulated manifolds or finite simplicial complexes. 

Let us finally recall that topological handle decompositions exist on all closed topological manifolds except the non smooth four-dimensional ones, see \cite{KS,FQ}. Do the latter carry any topological pinched handle decomposition, that is any filtration $\emptyset = M_0 \subset M_1 \subset \dots \subset M_N=M$ by topological spaces such that for every $p \in \{ 1, \dots , N \}$, $M_p$ is obtained from $M_{p-1}$ by attaching either a basic tile or a pinched handle along their missing faces using some homeomorphism?

\subsection{Morse shellings of subdivided tiles}
\label{subsecsubtiles}

We already proved \cite{SaWel1,SaWel2} that the first barycentric subdivision of any basic (resp. Morse) tile $T$ carries shellable $h$-tilings (resp. Morse tilings) which use
a critical tile if and only if $T$ is critical and this tile is then unique, isomorphic to $T$, see also \cite{WelAdv}. In order to prove Theorem \ref{maintheo}, we need the following variant of this result,  proven in a different way. 

\begin{theo}
\label{theotile1}
Let $T$ and $T'$ be two non empty basic tiles. 

1) If $T$ is an open simplex and $T'$ a closed one, then $\sd (T * T')$ carries a Morse shelling which begins with $N(T , T * T')$. Moreover, the latter uses a unique critical tile
in $N(T , T * T')$, of index $\dim (T)$, and a unique critical tile outside $N(T , T * T')$, of index $\dim (T) +1$.

2) Otherwise, $\sd (T * T')$ carries a shelled $h$-tiling which begins with $N(T , T * T')$ and uses a critical tile if and only if $T* T'$ is critical. Moreover, this tile is then unique, isomorphic to $T* T'$  and belongs to $N(T , T * T')$ if $T* T'$ is a closed simplex while it lies outside this neighborhood if $T* T'$ is an open one.
\end{theo}
The $h$-tiling given by the second part of Theorem \ref{theotile1} uses only basic tiles, it thus makes the above mentioned result of  \cite{SaWel1} more precise. By the way, the latter applies in the case covered by the first part of  Theorem \ref{theotile1} as well, providing a shelled $h$-tiling using no critical tile but which does not begin with $N(T , T * T')$, so that the two critical tiles of consecutive indices given by Theorem \ref{theotile1} can be cancelled. 

\begin{theo}
\label{theotile2}
Let $T$ be a non empty basic tile and $T'$ a non basic Morse tile. Then, $\sd (T * T')$ carries a Morse shelling which begins with $N(T , T * T')$ and uses a critical tile if and only if $T* T'$ is critical. Moreover, this tile is then unique, isomorphic to $T* T'$ and disjoint from $N(T , T * T')$.
\end{theo}

\begin{proof}[Proof of Theorem \ref{theotile1}]
We proceed by induction on $n = \dim (T) + \dim (T')$.  If $n=0$, both $T$ and $T'$ are zero-dimensional and four cases have to be considered, depending on whether they are closed or open vertices, that is vertices deprived of their empty face. One checks the result in these cases, see Figure \ref{figtilings}, and when $T$ is an open simplex and $T'$ a closed one, the tiling obtained on the subdivided edge $\sd (T * T')$ uses an open simplex together with a closed one deprived of its empty face. The latter is not a basic tile.

 \begin{figure}[h]
   \begin{center}
    \includegraphics[scale=1]{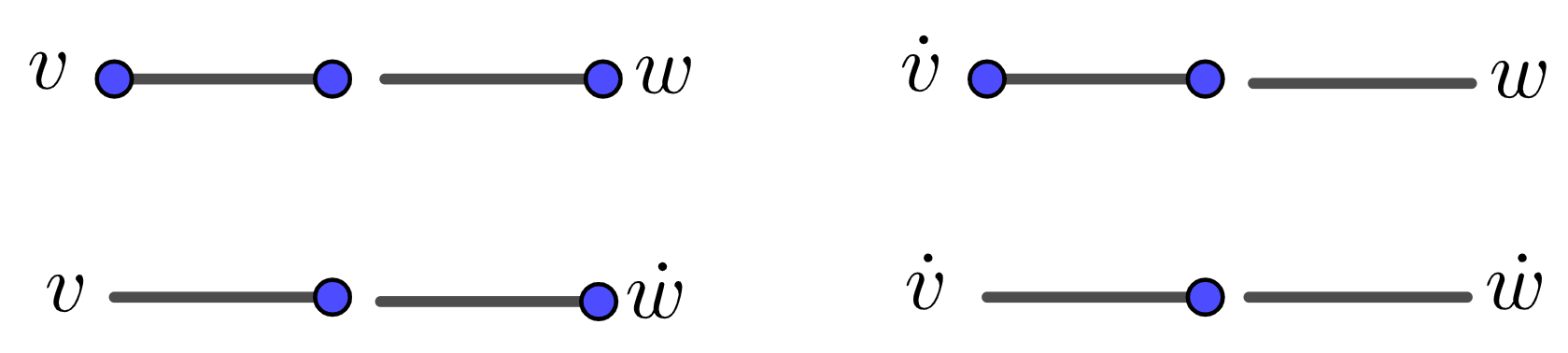}
    \caption{Tilings of $T* T'$ in dimension one.}
    \label{figtilings}
      \end{center}
 \end{figure}

Let us now assume the result proven up to the rank $n-1$ and prove it for $\dim (T) + \dim (T')=n$. By Proposition \ref{prop1}, $T = \sigma * \otheta$ and $T' = \sigma' * \stackrel{\circ}{\theta'}$ for some simplices  $\sigma, \sigma', \theta, \theta'$.  Let us number $v_0, \dots , v_k$ the vertices of $\overline{T} = \sigma * \theta$, beginning with those of $\sigma$ and then $v_{k+1}, \dots , v_{n+1}$ the vertices of $\overline{T}' = \sigma' * \theta'$, beginning with those of $\sigma'$. We deduce from Corollary \ref{cor3} a filtration $\emptyset \subset L_0 \subset L_1 \subset \dots \subset L_{n+1} = \sd (T*T')$ of relative simplicial complexes, where for every $j \in \{ 0 , \dots , n+1 \}$, $L_j = N([v_0, \dots , v_j] , T*T') = \cup_{i=0}^j \st_{ \sd (T*T')} (\hat{v}_i)$. We are going to Morse shell $L_{n+1}$ by finite induction on $j \in \{ 0 , \dots , n+1 \}$. By Corollary \ref{cor3}, $\lk_{ \sd (T*T')} (\hat{v}_0)$ is isomorphic to $\sd (\lk_{T*T'} (v_0))$ and by the induction hypothesis, the latter carries a shelled $h$-tiling which uses a critical tile if and only if the basic tile $\lk_{T*T'} (v_0)$ is critical and in this case, the critical tile is unique and isomorphic to $\lk_{T*T'} (v_0)$. In particular, this $h$-tiling uses a closed simplex iff $\lk_{T*T'} (v_0)$ is a closed simplex. By Corollary \ref{cor1}, $L_0 = \hat{v}_0 * \lk_{\sd (T*T')} (\hat{v}_0)$ inherits a shelled $h$-tiling which uses a critical tile if and only if $\lk_{T*T'} (v_0)$ is a closed simplex  and is this case, the critical tile is unique of vanishing index. Moreover, the latter is a closed simplex if $T*T'$ is a closed simplex and closed simplex deprived of its empty face otherwise, that is, by our choice of numbering of vertices, if $T = \dot{v}_0$ and $T'$ is a closed simplex.

Let us now assume $L_{j-1}$ equipped with a Morse shelling and let us extend it to $L_j$, $j \in \{ 1 , \dots , n \}$. The tile $\lk_{T*T'} (v_j)$ is a join $T_j * T'_j$, where the simplex $\overline{T}_j$ (resp. $\overline{T}'_j$ ) underlying the basic tile $T_j$ (resp. $T'_j$) has vertices $v_0, \dots , v_{j-1}$  (resp. $v_{j+1}, \dots , v_{n+1}$). By Corollary \ref{cor3}, $\st_{\sd (T*T')} (\hat{v}_j) \cap L_{j-1} $
coincides with $N(T_j , T_j * T'_j)$ in $\lk_{\sd (T*T')} (\hat{v}_j) $ via the isomorphism $\lk_{\sd (T*T')} (\hat{v}_j) \cong \sd (\lk_{T*T'} (v_j))$. By the induction hypothesis, two cases may occur. If $T_j$ is an open simplex and $T'_j$ a closed simplex, $\sd (T_j * T'_j)$ carries a Morse shelling which begins with $N(T_j , T_j * T'_j)$ and uses a critical tile of index $\dim (T_j)$ in this neighborhood and no closed simplex outside. By our choice of numbering of the vertices, this case may occur only if $T = T_j * \dot{v}_j$ and $T' = T'_j$ or if $T=T_j$ and $T' = v_j * T'_j$. The cone with apex $\hat{v}_j$ over this Morse shelling of $\sd (T_j * T'_j)$ provides a Morse shelling on $\st_{\sd (T*T')} (\hat{v}_j) = \hat{v}_j * \lk_{\sd (T*T')} (\hat{v}_j)$ and we extend the Morse shelling of 
$L_{j-1}$ to $L_j$ by concatenation of the latter, depriving however all cones with apex $\hat{v}_j$ over the tiles of $N(T_j , T_j * T'_j)$ of their basis. By Corollary \ref{cor1}, this extension adds exactly one critical tile, of index $\dim (T_j) +1$. When $T = T_j * \dot{v}_j$, the latter belongs to $N(T , T * T')$ and its index equals $\dim (T)$, whereas when $T=T_j$, the latter lies outside 
 $N(T , T * T')$ and its index equals $\dim (T)+1$.
 
 After this finite induction, we get a Morse shelling on $L_n$ which we now have to extend to $L_{n+1}$. By the induction hypothesis and Corollary \ref{cor3}, $\lk_{\sd (T*T')} (\hat{v}_{n+1}) \cong
 \sd (\lk_{T*T'} (v_{n+1}))$ carries a shelled $h$-tiling which uses only basic tiles among which an open simplex if and only if $\lk_{T*T'} (v_{n+1})$ is an open simplex. By concatenation of the cone with apex $\hat{v}_{n+1}$ over this $h$-tiling and deprived of its base, we get the Morse shelling of $L_{n+1}$.  By Corollary \ref{cor1}, this extension uses only basic tiles and adds a critical one if and only if $\lk_{T*T'} (v_{n+1})$ is an open simplex. By our choice of numbering of the vertices, this case occurs when either $T*T'$ is an open simplex, or $T$ is an open simplex and $T'$ a closed simplex reduced to $v_{n+1}$. The result follows.
\end{proof}

\begin{proof}[Proof of Theorem \ref{theotile2}]

The proof goes along the same lines as the one of Theorem \ref{theotile1} and proceeds by induction on $n = \dim (T) + \dim (T')$. The minimal dimension for which there exists a non basic Morse tile is one and the latter is then unique, isomorphic to a closed simplex deprived of its empty face. If $n=1$ then, there are to cases to consider, depending on whether $T$ is a closed vertex or an open one, that is deprived of its empty face. Set $T=v$ in the first case and $T= \dot{v}$ in the second one, whereas $T' = \dtau$ for some closed edge $\tau$. By Theorem \ref{theotile1}, $\sd (v * \tau)$ carries a shelled $h$-tiling  which begins with $N (v , v* \tau)$ and uses a closed simplex as unique critical tile, the latter belonging to $N (v , v* \tau)$ and thus containing ${v}$. By removing ${v}$, we get a Morse shelling on $\sd (v * T')$ which begins with $N (v , v* T')$ and does not use any critical tile. Likewise, Theorem \ref{theotile1} provides a Morse shelling on  $\sd (\dot{v} * \tau)$ which begins with $N (\dot{v} , \dot{v}* \tau)$ and uses a unique critical tile of vanishing index in this neighborhood and a unique critical tile of index one outside. The index zero critical tile is a closed simplex deprived of its empty face which contains the vertex $\hat{v} $ deprived of its empty face. By removing this vertex, we get a Morse shelling on $\sd (\dot{v} * T')$ which begins with $N (\dot{v}  , \dot{v} * T')$ and uses a unique critical tile of index one, thus isomorphic to $\dot{v} * T'$, and located outside this derived neighborhood.

Let us now assume the result proven up to the rank $n-1$ and prove it for $ \dim (T) + \dim (T') = n$. By Proposition \ref{prop1}, $T = \sigma * \otheta$ and $T' = \sigma' * \stackrel{\circ}{\theta'} * \dtau'$ for some simplices  $\sigma, \sigma', \theta, \theta', \tau'$.  Let us number $v_0, \dots , v_k$ the vertices of $\overline{T} = \sigma * \theta$, beginning with those of $\sigma$ and then $v_{k+1}, \dots , v_{n+1}$ the vertices of $\overline{T}' = \sigma' * \theta' * \tau'$, beginning with those of $\sigma'$, then those of $\theta'$ and ending with those of $\tau'$. We deduce from Corollary \ref{cor3} a filtration $\emptyset \subset L_0 \subset L_1 \subset \dots \subset L_{n+1} = \sd (T*T')$ of relative simplicial complexes, where for every $j \in \{ 0 , \dots , n+1 \}$, $L_j = N([v_0, \dots , v_j] , T*T') = \cup_{i=0}^j \st_{ \sd (T*T')} (\hat{v}_i)$. We are again going to Morse shell $L_{n+1}$ by finite induction on $j \in \{ 0 , \dots , n+1 \}$. By Corollary \ref{cor3}, $\lk_{ \sd (T*T')} (\hat{v}_0)$ is isomorphic to $\sd (\lk_{T*T'} (v_0))$ and $\lk_{T*T'} (v_0)$ is either a closed simplex deprived of its empty face, or a Morse tile for which the induction hypothesis holds. In both cases, $\sd (\lk_{T*T'} (v_0))$ carries a Morse shelling which does not use any closed simplex and by Corollary \ref{cor1}, $L_0 = \hat{v}_0 * \lk_{\sd (T*T')} (\hat{v}_0)$ inherits a Morse shelling which does not use any critical tile.
Let us now assume $L_{j-1}$ equipped with a Morse shelling and extend it to $L_j$, $j \in \{ 1 , \dots , n \}$. The tile $\lk_{T*T'} (v_j)$ is a join $T_j * T'_j$, where the simplex $\overline{T}_j$ (resp. 
$\overline{T}'_j$ ) underlying the basic tile $T_j$ (resp. $T'_j$) has vertices $v_0, \dots , v_{j-1}$  (resp. $v_{j+1}, \dots , v_{n+1}$). By Lemma \ref{lemma2}, the tile $T_j$ is basic whereas $T'_j$ is basic if $v_j$ is a vertex of $\tau'$ and Morse non basic otherwise. By Corollary \ref{cor3}, $\st_{\sd (T*T')} (\hat{v}_j) \cap L_{j-1} $
coincides with $N(T_j , T_j * T'_j)$ in $\lk_{\sd (T*T')} (\hat{v}_j) $ via the isomorphism $\lk_{\sd (T*T')} (\hat{v}_j) \cong \sd (\lk_{T*T'} (v_j))$. If $v_j$ is not the first vertex of $\tau'$ or if $T*T'$ is not critical, the induction hypothesis or Theorem \ref{theotile1} depending on the case provides a Morse shelling on  $\sd (T_j * T'_j)$ which begins with $N(T_j , T_j * T'_j)$ and uses no critical tile in this neighborhood except possibly a closed simplex. The cone with apex $\hat{v}_j$ over this Morse shelling provides a Morse shelling on $\st_{\sd (T*T')} (\hat{v}_j) = \hat{v}_j * \lk_{\sd (T*T')} (\hat{v}_j)$ and we extend the Morse shelling of 
$L_{j-1}$ to $L_j$ by concatenation of the latter, depriving however all cones with apex $\hat{v}_j$ over the tiles of $N(T_j , T_j * T'_j)$ of their basis. By Corollary \ref{cor1}, this extension does not add any critical tile. If $v_j$ is the first vertex of $\tau'$ and if $T*T'$ is critical, so that $T=T_j$ is an open simplex and $T'_j$ a closed one, Theorem \ref{theotile1} provides a Morse shelling on $\sd (T_j * T'_j)$ which begins with $N(T_j , T_j * T'_j)$ and uses a unique critical tile of index $\dim (T)$ in this neighborhood. The cone with apex $\hat{v}_j$ over this Morse shelling provides a Morse shelling on $\st_{\sd (T*T')} (\hat{v}_j) $ and we extend the Morse shelling of 
$L_{j-1}$ to $L_j$ by concatenation of the latter, depriving however all cones with apex $\hat{v}_j$ over the tiles of $N(T_j , T_j * T'_j)$ of their basis. By Corollary \ref{cor1}, this extension adds one critical tile of index $\dim (T) + 1$, so that it is isomorphic to $T*T'$. 
We have now extended the Morse shelling up to $L_n$ by using a critical tile if and only if $T*T'$ is critical and this tile is then unique, disjoint from $N(T, T*T')$. But we know from Theorem \ref{theotile1} that $\lk_{\sd (T*T')} (\hat{v}_{n+1}) \cong
 \sd (\lk_{T*T'} (v_{n+1}))$ carries a shelled $h$-tiling which uses only basic tiles and no open simplex since $\lk_{\tau'} (v_{n+1})$ is a non empty closed simplex. By concatenation of the cone with apex $\hat{v}_{n+1}$ over this $h$-tiling and deprived of its base, we get the Morse shelling of $L_{n+1} = \sd (T*T')$ without adding any critical tile by Corollary \ref{cor1}. Hence the result.
\end{proof}

\subsection{Morse shellings on first barycentric subdivisions}
\label{subsecfirst}

We already proved in \cite{WelAdv} that the first barycentric subdivision of every relative simplicial complex carries Morse shellings. We however need the following slightly refined version.

\begin{theo}
\label{theofirst}
Let $S=K \setminus L$ be a finite relative simplicial complex and let $v$ be a vertex of $K$. Then, $\sd (S)$ carries a Morse shelling which begins with $\st_{\sd (S)} (\hat{v}) = N(v,S)$.
\end{theo}

\begin{proof}
Let us number $\sigma_1, \dots, \sigma_N$ the facets of $K$, begining with those of $\st_K (v)$, say $\sigma_1, \dots, \sigma_k$. It induces a filtration $\emptyset = K_0 \subset K_1 \subset \dots \subset K_N = K$ of subcomplexes, where for every $j \in \{ 1, \dots , N \}$, $K_j$ denotes the complex containing $\sigma_1, \dots, \sigma_j$ together with their faces. We then set $P_j = K_j \setminus (K_{j-1} \cup L)$. This relative facet can be written $T_j \setminus \tau_j$, where $T_j$ is a basic tile and $\tau_j$ a union of faces of $T_j$ of codimension greater than one. Moreover, 
by construction, for every $j \in  \{ 1, \dots , k \}$, $T_j = v * T'_j$ where $T'_j$ is the ridge of $T_j$ opposite to $v$, which is contained in $\lk_K (v)$ and may be deprived of its empty face. By Theorem \ref{theotile1}, $\sd (T_j)$ carries then a Morse shelling  which begins with $N(v , T_j)$ and does not use any open simplex. We now proceed as in the proof of Proposition $3.3$ of \cite{WelAdv}. 
Let us denote by $T_j^1 , \dots , T_j^{(n_j+1)!}$ the tiles of this $h$-tiling following the shelling order. The closed  simplex $\overline{T}_j^p$ underlying $T_j^p$ reads $\{ {\sigma}_0^p, \dots , {\sigma}_{n_j}^p \}$, where for every $0 \leq l \leq m \leq n_j$, $\sigma_m^p$ is a $m$-dimensional face of $\overline{T}_j^p$ containing $\sigma_l^p$. For every $p \in \{1 , \dots , (n_j+1)!\}$, such that $T_j^p$ intersects $\sd (\tau_j)$, let us denote by $i_p$ the greatest element in $\{ 0 , \dots , n_j \}$ such that ${\sigma}_{i_p}^p$ is contained in $\overline{\tau}_j$. Then, $\sigma_{i}^p$ is contained in $\overline{\tau}_j$ for every $0 \leq i \leq i_p$ and moreover $i_p < n_j-1$ by assumption. We deduce that ${T}_j^p \cap \sd(\tau_j)$ coincides with the face ${T}_j^p \cap [{\sigma}_0^p, \dots , {\sigma}_{i_p}^p]$, so that ${T}_j^p \setminus \sd(\tau_j)$ is a Morse tile. Hence, for every $j \in \{ 1, \dots , k \}$, $\sd (P_j)$ carries a Morse shelling which begins with $N (v , P_j)$. Likewise, by Proposition $3.3$ of \cite{WelAdv},  for every $j \in \{ k+1, \dots , N \}$, $\sd (P_j)$ carries a Morse shelling. After concatenation of the shellings of  $N (v , P_j)$ for every $j \in \{ 1, \dots , k \}$, we deduce a Morse shelling of $N (v,S)$ which we extend to a Morse shelling of $\sd (\st_S (v))$ by concatenation of what remains of the shellings of $\sd (P_j)$, $j \in \{ 1, \dots , k \}$. We finally extend the Morse shelling of $\sd (\st_S (v))$ to the whole $\sd (S)$ by concatenation of the Morse shellings of $\sd (P_j)$, $j \in \{ k+1, \dots , N \}$. Hence the result.
\end{proof}

\subsection{Proof of Theorem \ref{maintheo}}
\label{subsecmaintheo}

By section $3$ of \cite{For1} we know that perturbing the discrete Morse function $f$ a bit if necessary, we may assume that its sub level sets $\{ \sigma \in K \; \vert \; f(\sigma) \leq m \}$, $m \in \R$, induce a filtration $\emptyset = K_0 \subset K_1 \subset \dots \subset K_N = K$ of $K$ such that for every $i \in \{ 1, \dots , N \}$, $K_i \setminus K_{i-1}$ consists either of a single simplex $\sigma_i$ which is then critical for $f$, or of a pair of simplices $\theta_i, \tau_i$, regular for $f$, such that $\theta_i$ is a ridge of $\tau_i$, so that $K_i$ collapses to $K_{i-1}$. This filtration induces
 a filtration $\emptyset = L_0 \subset L_1 \subset \dots \subset L_N = \sd^2 (K)$ of $\sd^2 (K)$, where for every $i \in \{ 1, \dots , N \}$, $L_i = N ( \sd (K_i) , \sd (K)) = \cup_{\sigma \in K_i} \st_{\sd^2 (K)} (\hat{\hat{\sigma}})$. We are going to prove that any Morse shelling on $L_{i-1}$ extends to $L_i$ and moreover, this extension can be chosen to use a critical tile if and only if $K_i \setminus K_{i-1} = \{ \sigma_i \}$ is a critical face of $f$ and this critical tile is then unique, of the same index $\dim (\sigma_i)$ as $\sigma_i$. The result then follows by finite induction, since nothing has to be proven for $i=1$.
 Let thus $i \in \{ 1, \dots , N \}$ and $L_{i-1}$ be equipped with a Morse shelling. Two cases have to be considered. If $K_i \setminus K_{i-1} = \{ \sigma_i \}$ is a critical level of $f$, then
 $L_i = L_{i-1} \cup \st_{\sd^2 (K)} (\hat{\hat{\sigma}}_i)$. If $\dim (\sigma_i) = 0$, this union is disjoint. Then, by Corollary \ref{cor2}, $\lk_{\sd^2 (K)} (\hat{\hat{\sigma}}_i)$ is isomorphic to 
 $\sd^2 (\lk_K (\sigma_i))$ and by Theorem \ref{theofirst}, in fact Theorem $1.3$ of \cite{WelAdv}, $\sd^2 (\lk_K (\sigma_i))$ carries a Morse shelling which uses a unique closed simplex containing the empty face, the other possible critical tiles of vanishing index being deprived of the empty face. By Corollary \ref{cor3}, $ \st_{\sd^2 (K)} (\hat{\hat{\sigma}}_i) = \hat{\hat{\sigma}}_i * \lk_{\sd^2 (K)} (\hat{\hat{\sigma}}_i)$ inherits a Morse shelling which uses a unique critical, a closed simplex. We then obtain by concatenation a Morse shelling on $L_i = L_{i-1} \sqcup \st_{\sd^2 (K)} (\hat{\hat{\sigma}}_i)$ by removing to this closed simplex its empty face if $L_{i-1}  \neq \emptyset$. In all cases, this extension adds one critical tile, with vanishing index. Now, if $\dim (\sigma_i) \neq 0$, the simplex $\sigma_i$ is maximal in $K_i$ and by Corollary \ref{cor2}, $\lk_{\sd^2 (K)} (\hat{\hat{\sigma}}_i)$ is isomorphic to $\sd \big( \sd (\partial \sigma_i) * \sd (\lk_K (\sigma_i)) \big)$. Moreover, through this isomorphism, 
 \begin{equation}
 \label{eqn1}
 \st_{\sd^2 (K)} (\hat{\hat{\sigma}}_i) \cap L_{i-1}  \cong  N \big( \sd (\partial \sigma_i) ,  \sd (\partial \sigma_i) * \sd (\lk_K (\sigma_i))  \big)
 \end{equation}

The complex $ \sd (\partial \sigma_i)$ is shellable since $\partial \sigma_i$ is by \cite{Bjo} and we set $\sd (\partial \sigma_i) = T_0 \sqcup \dots \sqcup T_a$ the associated shelled $h$-tiling, where $T_0$ is a closed simplex, $T_a$ an open one and the remaining tiles are basic and regular. Likewise, $\sd (\lk_K (\sigma_i))$ carries a Morse shelling by Theorem \ref{theofirst}, in fact Theorem $1.3$ of \cite{WelAdv},  and we set  $\sd (\lk_K (\sigma_i)) = T'_0 \sqcup \dots \sqcup T'_b$, where $T'_0$ is the unique closed simplex of the Morse shelling, containing the empty face, the other possible critical tiles of vanishing index being deprived of the empty face. By Theorems \ref{theotile1} and \ref{theotile2}, for every $(l,m) \in \{ 0, \dots , a \} \times \{ 0, \dots , b \}$, $\sd (T_l * T'_m)$ carries a Morse shelling which begins with $N(T_l , T_l * T'_m)$. Moreover, it uses a critical tile in this neighborhood if and only if $m=0$ and $l \in \{0 , a \}$, this critical tile being unique and isomorphic to a closed simplex if $l=0$ and a critical tile of index $\dim (\partial \sigma_i)$ if $l=a$. By concatenation of these Morse shellings of $N(T_l , T_l * T'_m)$, where $(l,m) \in \{ 0, \dots , a \} \times \{ 0, \dots , b \}$ follows the lexicographic order, and then by concatenation of what remains of the Morse shellings of $\sd (T_l * T'_m)$, still following the lexicographic order of $\{ 0, \dots , a \} \times \{ 0, \dots , b \}$, we get a Morse shelling of $\lk_{\sd^2 (K)} (\hat{\hat{\sigma}}_i)$ which begins with $N \big( \sd (\partial \sigma_i) ,  \sd (\partial \sigma_i) * \sd (\lk_K (\sigma_i))  \big)$ under the isomorphism (\ref{eqn1}). Moreover, the only critical tiles involved in the tiling of this derived neighborhood are one closed simplex and one critical tile of index $\dim (\partial \sigma_i)$. The cone with apex $\hat{\hat{\sigma}}_i$ over this Morse shelling of $\lk_{\sd^2 (K)} (\hat{\hat{\sigma}}_i)$ provides a Morse shelling of $\st_{\sd^2 (K)} (\hat{\hat{\sigma}}_i)$ and we extend the Morse shelling of $L_{i-1}$ to $L_i$ by concatenation of the latter, depriving however the cones with apex $\hat{\hat{\sigma}}_i$ over the tiles of $N \big( \sd (\partial \sigma_i) ,  \sd (\partial \sigma_i) * \sd (\lk_K (\sigma_i))  \big)$ of their basis. By Corollary \ref{cor1}, this extension from $L_{i-1}$ to $L_i$ adds a unique critical tile, of index $\dim (\sigma_i)$. 

Now, if $K_i \setminus K_{i-1} = \{ \theta_i, \tau_i \}$, where  $\theta_i$ is a ridge of $\tau_i$, we set $M_i = L_{i-1} \cup \st_{\sd^2 (K)} (\hat{\hat{\tau}}_i)$, so that  $L_i = M_{i} \cup \st_{\sd^2 (K)} (\hat{\hat{\theta}}_i)$. Again, the simplex $\tau_i$ is maximal in $K_i$, $\lk_{\sd^2 (K)} (\hat{\hat{\tau}}_i)$ is isomorphic to $\sd \big( \sd (\partial \tau_i) * \sd (\lk_K (\tau_i)) \big)$ by Corollary \ref{cor2} and through this isomorphism, 
 \begin{equation}
 \label{eqn2}
 \st_{\sd^2 (K)} (\hat{\hat{\tau}}_i) \cap L_{i-1}  \cong  N \big( \sd (\partial \tau_i \setminus \theta_i) ,  \sd (\partial \tau_i) * \sd (\lk_K (\tau_i))  \big),
 \end{equation}
 where $\partial \tau_i \setminus \theta_i$ denotes the simplicial complex $\partial \tau_i $ deprived of the facet $ \theta_i$.
The complex $\partial \tau_i$ carries a classical shelling which ends by the facet $\theta_i$ and induces a shelling on $ \sd (\partial \tau_i)$ by \cite{Bjo}. We set $\sd (\partial \tau_i) = T_0 \sqcup \dots \sqcup T_c$ the associated shelled $h$-tiling and $c' \in \{ 0 , \dots , c \}$ the integer such that  $T_0 \sqcup \dots \sqcup T_{c'}$ provides a shelling of $ \sd (\partial \tau_i \setminus \theta_i)$.
Thus, $T_{c'+1} \sqcup \dots \sqcup T_{c}$ provides a shelled $h$-tiling of the first barycentric subdivision of the open simplex $\otheta_i$, so that this shelling is a cone with apex $\hat{\theta}_i$ and deprived of its base over a shelling of $\sd (\partial \theta_i)$. For every $j \in \{ c'+1, \dots , c \} $, we set
$T_j = \dot{\hat{\theta}}_i * T''_j$, where $T''_{c'+1} \sqcup \dots \sqcup T''_{c}$ shells $\sd (\partial \theta_i)$. Likewise, $\sd (\lk_K (\tau_i))$ carries a Morse shelling by Theorem \ref{theofirst}, in fact Theorem $1.3$ of \cite{WelAdv},  and we set  $\sd (\lk_K (\tau_i)) = T'_0 \sqcup \dots \sqcup T'_d$, where $T'_0$ is the unique closed simplex of the Morse shelling, containing the empty face, the other possible critical tiles of vanishing index being deprived of the empty face. By Theorems \ref{theotile1} and \ref{theotile2}, for every $(l,m) \in \{ 0, \dots , c' \} \times \{ 0, \dots , d \}$, $\sd (T_l * T'_m)$ carries a Morse shelling which begins with $N(T_l , T_l * T'_m)$ and does not use any critical tile in this neighborhood except a unique closed simplex when $l=m=0$.
In the same way,  for every $(l,m) \in \{ c'+1, \dots , c \} \times \{ 0, \dots , d \}$, $\sd (T_l * T'_m)$ carries a Morse shelling which begins with $N(T''_l , T_l * T'_m)$ and does not use any critical tile in this neighborhood, since $T_l * T'_m = T''_l *  \dot{\hat{\theta}}_i  *T'_m$ and none of the Morse tiles $\dot{\hat{\theta}}_i  *T'_m$  are closed simplices by Corollary \ref{cor1}. 
 By concatenation of these Morse shellings of $N(T_l , T_l * T'_m)$, where $(l,m) \in \{ 0, \dots , c' \} \times \{ 0, \dots , d \}$ follows the lexicographic order, and then by concatenation of 
 the Morse shellings of $N(T''_l , T_l * T'_m)$, where $(l,m) \in \{ c'+1, \dots , c \} \times \{ 0, \dots , d \}$ follows the lexicographic order and finally by concatenation of
 what remains of the Morse shellings of $\sd (T_l * T'_m)$  following the lexicographic order of $ \{ 0, \dots , c \} \times \{ 0, \dots , d \}$, we get a Morse shelling of $\lk_{\sd^2 (K)} (\hat{\hat{\tau}}_i)$ which begins with $N \big( \sd (\partial \tau_i \setminus \theta_i) ,  \sd (\partial \tau_i) * \sd (\lk_K (\tau_i))  \big)$ under the isomorphism (\ref{eqn2}). Moreover, the only critical tile involved in the tiling of this derived neighborhood is a closed simplex. The cone with apex $\hat{\hat{\tau}}_i$ over the Morse shelling of $\lk_{\sd^2 (K)} (\hat{\hat{\tau}}_i)$ provides a Morse shelling of $\st_{\sd^2 (K)} (\hat{\hat{\tau}}_i)$ and we extend the Morse shelling of $L_{i-1}$ to $M_i$ by concatenation of the latter, depriving however the cones with apex $\hat{\hat{\tau}}_i$ over the tiles of 
 $N \big( \sd (\partial \tau_i \setminus \theta_i) ,  \sd (\partial \tau_i) * \sd (\lk_K (\tau_i))  \big)$ of their basis. By Corollary \ref{cor1}, this extension from $L_{i-1}$ to $M_i$ does not add any critical tile to the Morse tiling. 
 
 It remains to extend the Morse shelling we just obtained on $M_i$ to the whole $L_i = M_{i} \cup \st_{\sd^2 (K)} (\hat{\hat{\theta}}_i)$, without using any critical tile. By Corollary \ref{cor2}, $\lk_{\sd^2 (K)} (\hat{\hat{\theta}}_i)$ is isomorphic to $\sd \big( \sd (\partial \theta_i) * \sd (\lk_K (\theta_i)) \big)$. The link $\lk_K (\theta_i)$ contains the vertex $\lk_{\tau_i} (\theta_i)$ and we denote with some abuse by $\hat{\tau}_i$ the associated vertex of $ \sd (\lk_K (\theta_i))$. By Theorem \ref{theofirst}, $\sd (\lk_K (\theta_i))$ carries a Morse shelling which begins with $\st_{\sd (\lk_K (\theta_i))} (\hat{\tau}_i)$ and we set  $\sd (\lk_K (\theta_i)) = T'_0 \sqcup \dots \sqcup T'_e$ such a Morse shelling and denote by $e' \in \{ 0, \dots , e \}$ the integer such that $T'_0 \sqcup \dots \sqcup T'_{e'}$ Morse shells $\st_{\sd (\lk_K (\theta_i))} (\hat{\tau}_i)$. The latter is thus a cone with apex $\hat{\tau}_i$ over a tiling of $\lk_{\sd (\lk_K (\theta_i))} (\hat{\tau}_i)$, so that for every
 $l \in \{ 0, \dots , e' \}$, $T'_l = \hat{\tau}_i * T'''_l$, where $T'''_0 \sqcup \dots \sqcup T'''_{e'}$ Morse shells  $\lk_{\sd (\lk_K (\theta_i))} (\hat{\tau}_i)$. By Corollary \ref{cor2}, 
 $\st_{\sd (\lk_K (\theta_i))} (\hat{\tau}_i) \cap M_i \cong N \big( \sd (\partial \theta_i )* \hat{\tau}_i ,  \sd (\partial \theta_i) * \sd (\lk_K (\theta_i))  \big)$, where $ \sd (\partial \theta_i)$ is equipped with the shelling $T''_{c'+1} \sqcup \dots \sqcup T''_{c}$ chosen above. By Theorems \ref{theotile1} and \ref{theotile2}, for every $(l,m) \in \{ c'+1, \dots , c \} \times \{ 0, \dots , e' \}$, $\sd (T''_l * T'_m)$ carries a Morse shelling which begins with $N(T''_l * \hat{\tau}_i , T''_l * T'_m)$ and does not use any critical tile in this neighborhood except a unique closed simplex when $l=c'+1$ and $m=0$.
Likewise,  for every $(l,m) \in \{ c'+1, \dots , c \} \times \{ e'+1, \dots , e \}$, $\sd (T''_l * T'_m)$ carries a Morse shelling which begins with $N(T''_l , T''_l * T'_m)$ and does not use any critical tile in this neighborhood, since none of the tiles $T'_m$, $m>e'$ can be a closed simplex even though some might be closed simplices deprived of their empty faces. 
 By concatenation of the Morse shellings of $N(T''_l * \hat{\tau}_i , T''_l * T'_m)$, where $(l,m) \in \{ c'+1, \dots , c \} \times \{ 0, \dots , e' \}$ follows the lexicographic order, and then by concatenation of the Morse shellings of $N(T''_l , T''_l * T'_m)$, where $(l,m) \in \{ c'+1, \dots , c \} \times \{ e'+1, \dots , e \}$ follows the lexicographic order and finally by concatenation of
 what remains of the Morse shellings of $\sd (T''_l * T'_m)$  following the lexicographic order of $ \{ c'+1, \dots , c \} \times \{ 0, \dots , e \}$, we get a Morse shelling of $\lk_{\sd^2 (K)} (\hat{\hat{\theta}}_i)$ which begins with $N \big( \sd (\partial \theta_i )* \hat{\tau}_i ,  \sd (\partial \theta_i) * \sd (\lk_K (\theta_i)) \big)$ under the isomorphism given by Corollary \ref{cor2}. Moreover, the only critical tile involved in the tiling of this derived neighborhood is a closed simplex. The cone with apex $\hat{\hat{\theta}}_i$ over the Morse shelling of $\lk_{\sd^2 (K)} (\hat{\hat{\theta}}_i)$ provides a Morse shelling of $\st_{\sd^2 (K)} (\hat{\hat{\theta}}_i)$ and we extend the Morse shelling of $M_{i}$ to $L_i$ by concatenation of the latter, depriving however the cones with apex $\hat{\hat{\theta}}_i$ over the tiles of $N \big( \sd (\partial \theta_i )* \hat{\tau}_i ,  \sd (\partial \theta_i) * \sd (\lk_K (\theta_i)) \big)$ of their basis. By Corollary \ref{cor1}, this extension does not add any critical tile to the Morse tiling. Hence the result. \qed
 
 \begin{rem}
 \label{finalremark}
 1) The proof of Theorem \ref{maintheo} actually provides a relative version of this theorem as well. Namely, if $f$ is a discrete Morse function on the relative finite simplicial complex $S = K \setminus L$, then the relative simplicial complex $\sd^2 (K) \setminus N ( \sd (L) , \sd (K))$ carries a Morse shelling whose critical tiles are in one-to-one correspondence with the critical faces of $f$, preserving the index.
 
 2) If $f$ is the trivial discrete Morse function on $K$, for which all simplices are critical, then  Theorem \ref{maintheo} provides a Morse shelling on $\sd^2 (K)$ whose critical tiles are in one-to-one correspondence with the simplices of $K$, preserving the index. When $K$ is a combinatorial manifold, this result may be compared with Proposition $6.9$ of \cite{RS}, up to which $\sd^2 (K)$ inherits in this case a $PL$-handle decomposition whose handles are in one-to-one correspondence with the simplices of $K$, preserving the index. 
 
 3) The existence of Morse shellings on all barycentric subdivisions of finite simplicial complexes has already been obtained in \cite{WelAdv}, but without control on the critical tiles used, see also Theorem \ref{theofirst}.
 
 4) Conversely, every Morse shelling on a finite simplicial complex encodes a class of discrete Morse functions whose critical faces are in one-to-one correspondence with the critical tiles of the Morse shelling, preserving the index, see \cite{SaWel2,WelHHA}.
 
 5) Every Morse shelling on a finite simplicial complex provides two spectral sequences which converge to the (co)homology of the complex and whose first pages are free graded modules spanned by the critical tiles of the tiling, see \cite{WelHHA}.
 
 6) K. Adiprasito and B. Benedetti have likewise proved that the second barycentric subdivision of any linear triangulation of a convex polytope is shellable in the classical sense, see \cite{AdipBen2}.
 \end{rem}

Univ Lyon, Universit\'e Claude Bernard Lyon 1, CNRS UMR 5208, Institut Camille Jordan, 43 blvd. du 11 novembre 1918, F-69622 Villeurbanne cedex, France

{\tt welschinger@math.univ-lyon1.fr.}
\end{document}